\documentclass[12pt]{amsart}

\usepackage{framed}
\usepackage{braket}
\usepackage{amsmath,amssymb,amsthm}
\usepackage{color}
\usepackage{bm}
\usepackage[all]{xy}
\usepackage{amscd}
\usepackage{graphicx}
\usepackage{ascmac}
\usepackage{BOONDOX-frak, mathrsfs}
\usepackage[top=30truemm,bottom=30truemm,left=25truemm,right=25truemm]{geometry}
\usepackage{mathtools}
\mathtoolsset{showonlyrefs=true}
\usepackage{tikz}
\usetikzlibrary{cd}

\makeatletter
\@addtoreset{equation}{section}

\makeatother

\newcommand{\Z}{\mathbb{Z}}

\newcommand{\F}{\mathbb{F}}

\newcommand{\bE}{\mathbb{E}}

\newcommand{\bN}{\mathbb{N}}

\newcommand{\cI}{\mathcal{I}}

\newcommand{\cL}{\mathcal{L}}

\newcommand{\ol}[1]{\overline{#1}}

\newcommand{\lrang}[1]{\langle #1 \rangle}

\DeclareMathOperator{\ord}{ord}

\usepackage[OT2,T1]{fontenc}
\DeclareSymbolFont{cyrletters}{OT2}{wncyr}{m}{n}
\DeclareMathSymbol{\Sha}{\mathalpha}{cyrletters}{"58}

\let\oldenumerate\enumerate
\renewcommand{\enumerate}{
   \oldenumerate
   \setlength{\itemsep}{1pt}
   \setlength{\parskip}{0pt}
   \setlength{\parsep}{0pt}
}
\let\olditemize\itemize
\renewcommand{\itemize}{
   \olditemize
   \setlength{\itemsep}{1pt}
   \setlength{\parskip}{0pt}
   \setlength{\parsep}{0pt}
}

\theoremstyle{plain}
\newtheorem{thm}{Theorem}[section]
\newtheorem{lem}[thm]{Lemma}

\newtheorem{prop}[thm]{Proposition}

\theoremstyle{definition}
\newtheorem{defn}[thm]{Definition}

\newtheorem{rem}[thm]{Remark}
\newtheorem{eg}[thm]{Example}

\usepackage{diagbox}
\usepackage{tikz}
\usetikzlibrary{shapes.geometric}
\usepackage{mathtools}
\usepackage{float}
\usepackage{booktabs}

\newcommand{\swid}[1]{\hspace{0.05em} #1 \hspace{0.05em}}

\DeclareMathOperator{\aug}{aug}
\DeclareMathOperator{\Pic}{Pic}

\makeatletter
\@namedef{subjclassname@2020}{\textup{2020} Mathematics Subject Classification}
\makeatother

\title
[Graph coverings with prescribed Iwasawa invariants]
{Construction of graph coverings with prescribed Iwasawa invariants}
\author[T.~Kataoka]{Takenori Kataoka}
\address{Department of Mathematics, Faculty of Science Division II, Tokyo University of Science.
1-3 Kagurazaka, Shinjuku-ku, Tokyo 162-8601, Japan}
\email{tkataoka@rs.tus.ac.jp}
\keywords{Iwasawa theory, Iwasawa invariants, Graph theory}
\subjclass[2020]{05C25 (Primary), 11R23}
\date{\today}
\begin{document}

\maketitle

%%%%%%%%%%%%%%%%%%%%%
\begin{abstract}
For a $\Z_p$-covering of connected graphs, an analogue of Iwasawa's class number formula describes the growth of the number of spanning trees in terms of Iwasawa $\lambda$- and $\mu$-invariants.
In this paper, we show that any pair $(\lambda, \mu)$ can be realized as the Iwasawa invariants of an unramified $\Z_p$-covering of a bouquet, provided that the necessary condition that $\lambda$ is odd is satisfied.
We further show that any pair $(\lambda, \mu)$, without a parity condition, can be realized if we allow ramified $\Z_p$-coverings.
\end{abstract}
%%%%%%%%%%%%%%%%%%%%%

%%%%%%%%%%%%%%%%%%%%%%
\section{Introduction}\label{sec:intro}
%%%%%%%%%%%%%%%%%%%%%%

We begin with a quick review of Iwasawa theory of graphs (see Section \ref{sec:prelim} for details).
Let $p$ be a fixed prime number.
We study a (possibly ramified) $\Z_p$-covering of connected graphs, that is, a tower of graph coverings
\[
X = X_0 \leftarrow X_1 \leftarrow X_2 \leftarrow \cdots
\]
such that each $X_n$ is a connected graph and $X_n \to X$ is a Galois covering with Galois group isomorphic to $\Z/p^n\Z$.
We symbolically write $X_{\infty}/X$ for this $\Z_p$-covering.

For each $n \geq 0$, let $\kappa_{X_n}$ denote the number of spanning trees of $X_n$.
Then the analogue of Iwasawa's class number formula shows that there are non-negative integers
\[
\lambda = \lambda(X_{\infty}/X),
\quad
\mu = \mu(X_{\infty}/X),
\]
and an integer $\nu$ such that
\[
\ord_p(\kappa_{X_n}) = \lambda n + \mu p^n + \nu
\]
holds for $n \gg 0$.
Here, $\ord_p$ denotes the $p$-adic valuation normalized by $\ord_p(p) = 1$.
This formula is proved independently by Gonet \cite[Theorem 1.1]{Gon22} and McGown--Valli\`{e}res \cite[Theorem 6.1]{MV24} if $X_{\infty}/X$ is unramified, and by Gambheera--Valli\`{e}res \cite[Theorem A]{GV24} in general.

The main theme of this paper is to determine the possible pairs
\[
(\lambda(X_{\infty}/X), \mu(X_{\infty}/X))
\]
as $X_{\infty}/X$ varies.

For a while, we focus on the case where $X_{\infty}/X$ is unramified.
Then this question is closely related to the one studied by Dion--Lei--Ray--Valli\`{e}res \cite{DLRV24}.
They observed that $\lambda(X_{\infty}/X)$ is odd as long as $p \geq 3$, and constructed many examples.
In this paper, we first show that $\lambda(X_{\infty}/X)$ must be odd, including the case $p = 2$ (see Section \ref{sec:nec}).
Then the main theorem reveals that this is the only restriction.

\begin{thm}\label{thm:main}
For any odd $\lambda \geq 1$ and any $\mu \geq 0$, there exists an unramified $\Z_p$-covering $X_{\infty}/X$ of connected graphs satisfying $\lambda(X_{\infty}/X) = \lambda$ and $\mu(X_{\infty}/X) = \mu$.
Indeed, we construct such a covering with $X$ a bouquet (i.e., a graph with a single vertex).
\end{thm}

This theorem will be proved in Section \ref{sec:bouquets}.

In Section \ref{sec:remark}, we present two related remarks, still focusing on bouquets.
In the first, we discuss the probability of given invariants in a certain sense.
Indeed, we observe that $\mu(X_{\infty}/X) = 0$ holds in almost all cases, and that the probability of $\lambda(X_{\infty}/X) = 2\lambda' - 1$ is $\cfrac{1}{p^{\lambda'-1}}\left(1 - \cfrac{1}{p} \right)$.
This seems to be a natural estimate.
In the second, we note that one can furthermore achieve prescribed $\mu_l$-invariants for all prime numbers $l \neq p$.

Finally, in Section \ref{sec:ram}, we consider ramified $\Z_p$-coverings $X_{\infty}/X$.
We observe that the parity condition depends on the ramification; for example, when $X$ has two vertices, one unramified and the other totally ramified in $X_{\infty}/X$, then $\lambda(X_{\infty}/X)$ must be even.
The second main theorem reveals that this is the only restriction.

\begin{thm}\label{thm:main_ram}
For any even $\lambda \geq 0$ and any $\mu \geq 0$, there exists a (ramified) $\Z_p$-covering $X_{\infty}/X$ of connected graphs satisfying $\lambda(X_{\infty}/X) = \lambda$ and $\mu(X_{\infty}/X) = \mu$.
Indeed, we construct such a covering with $X$ a graph with two vertices, one unramified and the other totally ramified in $X_{\infty}/X$.
\end{thm}

Combining Theorems \ref{thm:main} and \ref{thm:main_ram}, we realize all $(\lambda, \mu)$ without any parity condition.

\begin{rem}
Iwasawa theory of graphs is developing as an analogue of the classical Iwasawa theory (Iwasawa \cite{Iwa59}) concerning the ideal class groups of number fields.
Naturally, the theme of the present paper has a counterpart in the classical setting, namely, to determine the possible pairs
\[
(\lambda(K_{\infty}/K), \mu(K_{\infty}/K))
\]
as the $\Z_p$-extension $K_{\infty}/K$ of number fields varies.
Toward this question, important progress was made by Ozaki \cite{Oza04} and Fujii--Ohgi--Ozaki \cite{FOO06}.
Ozaki \cite[Theorem B]{Oza04} constructed examples with prescribed $\mu$-invariants with $\lambda = 0$ when $p \geq 3$, while Fujii--Ohgi--Ozaki \cite[Theorem 1]{FOO06} constructed examples with prescribed $\lambda$-invariants with $\mu = 0$ when $p = 2, 3, 5$.
However, in general, such constructions remain an open problem.
\end{rem}

%%%%%%%%%%%%%%%%%%%%%%
\section{Preliminaries}\label{sec:prelim}
%%%%%%%%%%%%%%%%%%%%%%

%%%%%%%%%%%%%%%%%%%%%%
\subsection{Iwasawa invariants}\label{ss:Iw_inv}
%%%%%%%%%%%%%%%%%%%%%%

We first recall the definition of the Iwasawa invariants for elements of the formal power series ring $\Z_p[[T]]$.
Let $f \in \Z_p[[T]] \setminus \{0\}$.
The $\mu$-invariant $\mu(f)$ is the largest integer such that
\[
f \in p^{\mu(f)} \Z_p[[T]]
\]
holds.
The $\lambda$-invariant $\lambda(f)$ is then the largest integer such that
\[
\ol{p^{-\mu(f)} f} \in T^{\lambda(f)} \F_p[[T]]
\]
holds, where the overline denotes the reduction modulo $p$.
Note that $\F_p[[T]]$ is a discrete valuation ring and $T$ is a uniformizer.

Let $\Gamma$ be a profinite group (written multiplicatively) that is isomorphic to $\Z_p$.
For $n \geq 0$, we set $\Gamma_n = \Gamma/\Gamma^{p^n} \simeq \Z/p^n\Z$.
The Iwasawa algebra is defined as the projective limit
\[
\Z_p[[\Gamma]] = \varprojlim_n \Z_p[\Gamma_n],
\]
which is also called the completed group ring.
Let $\gamma$ be a fixed topological generator of $\Gamma$.
Then we have a topological $\Z_p$-algebra isomorphism
\[
\Z_p[[\Gamma]] \simeq \Z_p[[T]],
\quad
\gamma \leftrightarrow 1 + T,
\]
which is often called Serre's isomorphism.
Using this isomorphism, for each $f \in \Z_p[[\Gamma]] \setminus \{0\}$, we define $\lambda(f)$ and $\mu(f)$ as the invariants of the associated power series.
These invariants are independent of the choice of the topological generator $\gamma$.

In this paper we also deal with elements of the power series $\Z_p[T]$ and the group ring $\Z_p[\Gamma]$.
It is convenient to describe their relations in a commutative diagram
\[
\xymatrix{
\Z_p[T] \ar@{^(->}[r] \ar@{^(->}[d] & \Z_p[\Gamma] \ar@{^(->}[d] \\
\Z_p[[T]] \ar[r]_{\simeq} & \Z_p[[\Gamma]]
}
\]
obtained by identifying $T$ with $\gamma - 1$.
More generally, for any commutative ring $A$ (e.g., $A = \Z$), we have 
\[
\xymatrix{
A[T] \ar@{^(->}[r] \ar@{^(->}[d] & A[\Gamma] \ar@{^(->}[d] \\
A[[T]] & A[[\Gamma]]
}
\]
again by $T = \gamma - 1$, where we set $A[[\Gamma]] = \varprojlim_n A[\Gamma_n]$ (the injectivity of the map $A[\Gamma] \to A[[\Gamma]]$ can be checked in the same way as claim (3) of the next lemma applied to $a=0$).
The following elementary lemma is useful to study the $\mu$-invariants.

\begin{lem}\label{lem:divis1}
Let $A$ be a commutative ring and let $a \in A$ be an element.
\begin{itemize}
\item[(1)]
For $f \in A[T]$, we have $f \in a A[T]$ if and only if $f \in a A[[T]]$.
\item[(2)]
For $f \in A[T]$, we have $f \in a A[T]$ if and only if $f \in a A[\Gamma]$.
\item[(3)]
For $f \in A[\Gamma]$, we have $f \in a A[\Gamma]$ if and only if $f \in a A[[\Gamma]]$.
\end{itemize}
\end{lem}

\begin{proof}
The ``only if'' parts are all clear and we have to show the ``if'' parts.

(1)
Suppose $f \in A[T] \cap aA[[T]]$.
Then the coefficients of $f$ are all in $aA$, from which we deduce $f \in aA[T]$.

(2)
We have $A[T] = A[1 + T] \simeq A[\gamma^{\bN}]$, where $\gamma^{\bN}$ denotes the submonoid of $\Gamma$ generated by $\gamma$ and $A[\gamma^{\bN}]$ denotes the monoid ring.
Then by the construction of monoid rings, $A[\gamma^{\bN}]$ is a direct summand of $A[\Gamma]$ as $A$-modules.
The claim follows from this.

(3)
Suppose $f \in A[\Gamma]$ goes to an element of $aA[[\Gamma]]$.
We write $f = \sum_i c_i \gamma_i$ as a finite sum, where $c_i \in A$ and the $\gamma_i \in \Gamma$ are distinct.
We can take $n \geq 0$ such that the $\gamma_i$ are still distinct modulo $\Gamma^{p^n}$.
We write $\ol{f} = \sum_i c_i \ol{\gamma_i}$ for the projection to $A[\Gamma_n]$.
The assumption $f \in aA[[\Gamma]]$ implies $\ol{f} \in aA[\Gamma_n]$, so we obtain $c_i \in aA$ for each $i$.
Therefore, we have $f \in a A[\Gamma]$.
\end{proof}

%%%%%%%%%%%%%%%%%%%%%%
\subsection{The analogue of Iwasawa's class number formula}\label{ss:Iw_graph}
%%%%%%%%%%%%%%%%%%%%%%

Let us review the formalism of graphs and $\Z_p$-coverings, and how to compute the $\lambda$- and $\mu$-invariants.
In this subsection we focus on the unramified case; the ramified case is postponed to Subsection \ref{ss:ram_basic}.
For details, one may refer to the original papers by Gonet \cite{Gon22} and McGown--Valli\`{e}res \cite{MV24}.
Subsequent papers by Valli\`{e}res including \cite{DLRV24} also give detailed explanation.
We mainly follow the convention in the author's article \cite{Kata_21}.

%%%%%%%%%%%%%%%%%%%%%%
\subsubsection{Graphs}
%%%%%%%%%%%%%%%%%%%%%%

By a graph we mean a finite multigraph, so the set of vertices and the set of edges are both finite, and multi-edges and loops are allowed.
More formally, we use Serre's formalism on graphs as follows.

A graph $X$ consists of $(V_X, \bE_X, \ol{\cdot}, s, t)$, where $V_X$ and $\bE_X$ are finite sets, $\ol{\cdot}$ is an involution of $\bE_X$ without fixed elements, and $s, t: \bE_X \to V_X$ are maps.
We impose the condition $s(\ol{e}) = t(e)$ and $t(\ol{e}) = s(e)$ for any $e \in \bE_X$.
The elements of $V_X$ and $\bE_X$ are called vertices and edges, respectively.
Each $e \in \bE_X$ is regarded as an edge that connects $s(e)$ to $t(e)$, and $\ol{e}$ is regarded as the opposite of $e$.

%%%%%%%%%%%%%%%%%%%%%%
\subsubsection{Unramified $\Z_p$-coverings} 
%%%%%%%%%%%%%%%%%%%%%%

Let $X$ be a graph.
Let $\Gamma$ be a profinite group, written multiplicatively, that is isomorphic to $\Z_p$.
A voltage assignment $\alpha: \bE_X \to \Gamma$ is by definition a map satisfying
\[
\alpha(\ol{e}) = \alpha(e)^{-1}
\]
for any $e \in \bE_X$.
Associated to such an $\alpha$, we construct the unramified tower of graphs
\[
X = X_0 \leftarrow X_1 \leftarrow X_2 \leftarrow \cdots
\]
as follows.
For each $n \geq 0$, we set $\Gamma_n = \Gamma/\Gamma^{p^n}$ and define the derived graph $X_n$ as
\[
V_{X_n} = \Gamma_n \times V_X,
\quad
\bE_{X_n} = \Gamma_n \times \bE_X
\]
and
\[
\ol{(\gamma, e)} = (\gamma \cdot \alpha(e), \ol{e}),
\quad
s((\gamma, e)) = (\gamma, s(e)),
\quad
t((\gamma, e)) = (\gamma \cdot \alpha(e), t(e))
\]
for each $(\gamma, e) \in \bE_{X_n}$.
For $n' \geq n$, the morphism $X_{n'} \to X_n$ in the tower is induced from the projection $\Gamma_{n'} \to \Gamma_n$.
For each $n$, the group $\Gamma_n$ naturally acts on $X_n$ over $X$.

As long as $X_n$ are all connected, the tower above is an unramified $\Z_p$-covering of connected graphs.
Conversely, every unramified $\Z_p$-covering of connected graphs can be constructed in this manner, up to isomorphism.

%%%%%%%%%%%%%%%%%%%%%%
\subsubsection{The formula}\label{sss:unr_ICNF}
%%%%%%%%%%%%%%%%%%%%%%

Associated to the voltage assignment $\alpha: \bE_X \to \Gamma$, we define 
\[
\cL_{\alpha}: \bigoplus_{v \in V_X} \Z[\Gamma] v \to \bigoplus_{v \in V_X} \Z[\Gamma] v
\]
as the $\Z[\Gamma]$-homomorphism satisfying
\[
\cL_{\alpha}(v) = \sum_{\substack{e \in \bE_{X} \\ s(e) = v}} (s(e) - \alpha(e) t(e))
\]
for each $v \in V_X$.
The matrix representation of $\cL_{\alpha}$ is called the Laplacian matrix.
We define
\[
Z_{\alpha} = \det(\cL_{\alpha}) \in \Z[\Gamma]
\]
as the determinant of the Laplacian matrix.

The following is the analogue of Iwasawa's class number formula with explicit description of the $\lambda$- and $\mu$-invariants.

\begin{thm}\label{thm:ICNF_ref}
Let $X$ be a connected graph.
Let $\alpha: \bE_X \to \Gamma$ be a voltage assignment.
Let $X_{\infty}/X$ be the unramified $\Z_p$-covering constructed by the derived graphs, where we assume that each $X_n$ is connected.
Then
\[
\ord_p(\kappa_{X_n}) = \lambda n + \mu p^n + \nu
\]
holds for $n \gg 0$ with $\lambda = \lambda(Z_{\alpha}) - 1$, $\mu = \mu(Z_{\alpha})$, and some integer $\nu$.
\end{thm}

Here, $\lambda(Z_{\alpha})$ and $\mu(Z_{\alpha})$ are the $\lambda$- and $\mu$-invariants defined through the embedding $\Z[\Gamma] \subset \Z_p[[\Gamma]]$.
Note that $\mu$ is the largest integer satisfying $Z_{\alpha} \in p^{\mu} \Z_p[[\Gamma]]$, but by Lemma \ref{lem:divis1}(3) this is equivalent to $Z_{\alpha} \in p^{\mu} \Z_p[\Gamma]$, which is in turn equivalent to $Z_{\alpha} \in p^{\mu} \Z[\Gamma]$.
If we know $Z_{\alpha} \in \Z[T]$ under Serre's isomorphism (which depends on the choice of a topological generator $\gamma$), Lemma \ref{lem:divis1}(2) enables us to further rephrase this condition to $Z_{\alpha} \in p^{\mu} \Z[T]$.

%%%%%%%%%%%%%%%%%%%%%%
\section{The parity condition}\label{sec:nec}
%%%%%%%%%%%%%%%%%%%%%%

%%%%%%%%%%%%%%%%%%%%%%
\subsection{Statement and a reduction step}
%%%%%%%%%%%%%%%%%%%%%%

In this section, we show the following parity condition.

\begin{prop}\label{prop:odd}
For any unramified $\Z_p$-covering $X_{\infty}/X$ of connected graphs, $\lambda(X_{\infty}/X)$ is odd.
\end{prop}

This proposition is shown in \cite[Theorem 2.18]{DLRV24} in the case where $p \geq 3$.
The case $p = 2$ will require an additional technique (see Lemma \ref{lem:inv_even}).
In any case, the following proposition is the key to the proof.
Let $\Gamma$ be a profinite group isomorphic to $\Z_p$.
Let $\iota: \Z_p[[\Gamma]] \to \Z_p[[\Gamma]]$ be the involution that inverts every element of $\Gamma$.

\begin{prop}\label{prop:iota_inv_Z}
Let $X$ be a graph and $\alpha: \bE_X \to \Gamma$ a voltage assignment.
Then $\iota(\cL_{\alpha})$ is equal to the transpose of $\cL_{\alpha}$, and thus $\iota(Z_{\alpha}) = Z_{\alpha}$ holds.
\end{prop}

\begin{proof}
This follows from the definitions, as already shown in \cite[Theorem 2.18]{DLRV24}.
\end{proof}

This proposition leads us to study $\iota$-invariant elements of $\Z_p[[\Gamma]]$.
Indeed, Proposition \ref{prop:inv_g} claims that the $\lambda$-invariant of every $\iota$-invariant element is even.
Then Proposition \ref{prop:odd} follows since $\lambda(X_{\infty}/X) = \lambda(Z_{\alpha}) - 1$ by Theorem \ref{thm:ICNF_ref} with a suitable voltage assignment $\alpha$.

%%%%%%%%%%%%%%%%%%%%%%
\subsection{$\iota$-invariant power series}\label{ss:power_series}
%%%%%%%%%%%%%%%%%%%%%%

Let us study $\iota$-invariant power series.
Let $A$ be a commutative ring.
We define an $A$-algebra homomorphism
\[
\iota: A[[T]] \to A[[T]],
\quad
\iota(f) = f((1 + T)^{-1} - 1),
\]
noting that
\[
\iota(T) = (1+T)^{-1} - 1 = -T + T^2 - T^3 + \cdots \in T A[[T]].
\]
More illuminating property is that $\iota(1 + T) = (1 + T)^{-1}$, which also implies that $\iota$ is an involution.
Under Serre's isomorphism $\Z_p[[\Gamma]] \simeq \Z_p[[T]]$, this $\iota$ equals the $\iota$ that inverts every group element.

The set of $\iota$-invariant power series is denoted by
\[
A[[T]]^{\lrang{\iota}} := \{ f \in A[[T]] \mid \iota(f) = f\}.
\]
For a power series $f \in A[[T]] \setminus \{0\}$, its order $\ord_T(f)$ is defined as the largest integer $n$ such that $f \in T^n A[[T]]$.

\begin{lem}\label{lem:inv_even}
For any $f \in A[[T]]^{\lrang{\iota}} \setminus \{0\}$, the order $\ord_T(f)$ is even.
\end{lem}

\begin{proof}
For any odd integer $n \geq 1$, let us show that $f \in T^n A[[T]]$ implies $f \in T^{n+1} A[[T]]$, from which the lemma follows.
We may write
\[
f(T) = a_n T^n + a_{n+1} T^{n+1} + \cdots.
\]
Then 
\begin{align*}
\iota(f) 
& = f((1+T)^{-1} - 1)
= f(-T + T^2 - \cdots)\\
& = a_n(-T)^n + n a_n (-T)^{n-1} T^2 + a_{n+1} (-T)^{n+1} + \cdots,
\end{align*}
so $f = \iota(f)$ implies
\[
a_n = (-1)^n a_n,
\quad
a_{n+1} = (-1)^{n+1}(na_n + a_{n+1}).
\]
Since $n$ is odd, we obtain
\[
2 a_n = 0,
\quad
n a_n = 0.
\]
Since $n$ is odd again, these imply $a_n = 0$.
Therefore, we have $f \in T^{n+1}A[[T]]$.
\end{proof}

In this proof, if $2$ is a non-zero-divisor in $A$, we only have to look at the coefficients of $T^n$.

The following lemma will be also useful in the subsequent sections.

\begin{lem}\label{lem:inv_exp}
We have an $A$-algebra isomorphism
\[
A[[S]] \overset{\simeq}{\to} A[[T]]^{\lrang{\iota}}
\]
given by $S \mapsto T + \iota(T) = T^2 - T^3 + \cdots$.
In other words, every $\iota$-invariant power series can be written as a power series of $T + \iota(T)$ in a unique way.
\end{lem}

\begin{proof}
The well-definedness of the map and its injectivity are clear.
To show the surjectivity, let us take any $f \in A[[T]]^{\lrang{\iota}}$.
We shall construct $b_0, b_1, b_2, \dots \in A$ satisfying
\[
f(T) = b_0 + b_1(T+\iota(T)) + b_2(T+\iota(T))^2 + \cdots.
\]
The construction is inductive:
For each $k \geq 0$, the elements are constructed so that
\[
f_k(T) := f(T) - (b_0 + b_1(T+\iota(T)) + \dots + b_{k-1}(T+\iota(T))^{k-1})
\]
satisfies $\ord_T(f_k) \geq 2k$.

Let $k \geq 0$ and suppose that we have $b_0, b_1, \dots, b_{k-1} \in A$ such that $f_k(T)$ satisfies $\ord_T(f_k) \geq 2k$.
We define $b_k \in A$ to be the coefficient of $T^{2k}$ in $f_k(T)$, and this defines 
\[
f_{k+1}(T) = f_k(T) - b_k(T + \iota(T))^k.
\]
Then, since $T + \iota(T) = T^2 - T^3 + \cdots$, we have $\ord_T(f_{k+1}) \geq 2k+1$.
Moreover, since $f_{k+1}$ is also $\iota$-invariant, Lemma \ref{lem:inv_even} implies that $\ord_T(f_{k+1}) \geq 2k+2$.
This completes the induction.
\end{proof}

Now we specialize to $A = \Z_p$.
We not only show that $\lambda(f)$ is even for any $f \in \Z_p[[T]]^{\lrang{\iota}}$, but also obtain a refined description that will be used in the subsequent sections.

\begin{prop}\label{prop:inv_g}
For any $f \in \Z_p[[T]]^{\lrang{\iota}}$, let $g(S) \in \Z_p[[S]]$ be the unique power series satisfying $g(T + \iota(T)) = f(T)$ given in Lemma \ref{lem:inv_exp}.
Then we have
\[
\lambda(f) = 2 \lambda_S(g),
\quad
\mu(f) = \mu_S(g),
\]
where $\lambda_S(-)$ and $\mu_S(-)$ denote the $\lambda$- and $\mu$-invariants for the algebra $\Z_p[[S]]$.
In particular, $\lambda(f)$ is even.
\end{prop}

\begin{proof}
The formulas follow easily from $T + \iota(T) = T^2 - T^3 + \cdots$.
\end{proof}

This also completes the proof of Proposition \ref{prop:odd}.

%%%%%%%%%%%%%%%%%%%%%%
\section{Proof of Theorem \ref{thm:main}}\label{sec:bouquets}
%%%%%%%%%%%%%%%%%%%%%%

In this and the next sections, we focus on bouquets.

%%%%%%%%%%%%%%%%%%%%%%
\subsection{$\Z_p$-coverings of bouquets}\label{ss:basic_bouquet}
%%%%%%%%%%%%%%%%%%%%%%

We review formulas for the $\lambda$- and $\mu$-invariants for bouquets, which were essentially already obtained in earlier work, such as \cite{Val21}.
For a bouquet $X$ and a voltage assignment $\alpha: \bE_X \to \Gamma$, it is known that $X_n$ are all connected if and only if the image of $\alpha$ is not contained in $\Gamma^p$.

\begin{prop}\label{prop:bou_compu}
Let $X$ be a bouquet, so $V_X$ is a singleton.
Let $\alpha: \bE_X \to \Gamma$ be a voltage assignment whose image is not contained in $\Gamma^p$, and let $X_{\infty}/X$ be the derived unramified $\Z_p$-covering.
Let $\bE_X = \{e_1, \ol{e_1}, \dots, e_t, \ol{e_t} \}$ and we set
\[
h_{\alpha} = \sum_{i=1}^t (1-\alpha(e_i)) \in \Z[\Gamma].
\]
Then we have $\lambda(X_{\infty}/X) = \lambda(h_{\alpha} + \iota(h_{\alpha})) - 1$ and $\mu(X_{\infty}/X) = \mu(h_{\alpha} + \iota(h_{\alpha}))$.
\end{prop}

\begin{proof}
Since $V_X$ is a singleton, the Laplacian matrix is identified with the scalar
\[
\cL_{\alpha} 
= \sum_{i=1}^t (1-\alpha(e_i)) + \sum_{i=1}^t (1-\alpha(e_i)^{-1})
= h_{\alpha} + \iota(h_{\alpha}),
\]
to which $Z_{\alpha}$ is also equal.
Therefore, the proposition follows from Theorem \ref{thm:ICNF_ref}.
\end{proof}

\begin{rem}\label{rem:compar}
In \cite{DLRV24}, to study the $\lambda$- and $\mu$-invariants, we write $\alpha(e_i) = \gamma^{a_i}$ with $a_i \in \Z_p$ and then consider the binomial expansion
\begin{align*}
h_{\alpha} + \iota(h_{\alpha})
& = \sum_{i=1}^t (1 - (1 + T)^{a_i}) + \sum_{i=1}^t (1 - (1 + T)^{-a_i})\\
& = - \sum_{k=1}^{\infty} \sum_{i=1}^t \left( \binom{a_i}{k} + \binom{-a_i}{k} \right) T^k.
\end{align*}
The binomial coefficients are defined by
\[
\binom{a}{k} := \frac{a(a-1) \cdots (a-k+1)}{k!}.
\]
Then the basic strategy in \cite{DLRV24} is to investigate directly the $p$-adic valuations of the coefficients of $h_{\alpha} + \iota(h_{\alpha})$.
An obstruction to this strategy is the presence of the denominator $k!$ in the binomial coefficients, which leads to an additional assumption that $\lambda$ is small compared to $p$ (e.g., in \cite[Lemma 4.12]{DLRV24}).
The idea of the present paper is to avoid the binomial expansion by instead using an expansion in the variable $S = T + \iota(T)$, which makes the computation easier.
\end{rem}

For convenience, we give a name to the elements $h_{\alpha}$ when $\alpha$ varies.

\begin{defn}\label{defn:adm}
We say an element $h \in \Z[\Gamma]$ {\it admissible} if the following conditions hold.
\begin{itemize}
\item[(a)]
We have $\aug(h) = 0$ for the augmentation map $\aug: \Z[\Gamma] \to \Z$.
\item[(b)]
The coefficients of $h$, except for the identity element, are all non-positive.
\item[(c)]
We have $h \not \in \Z[\Gamma^p]$ (for the connectedness).
\end{itemize}
\end{defn}

Then each $h_{\alpha}$ is admissible and, conversely, admissible elements are realized as $h_{\alpha}$.
Our task is now to study the Iwasawa invariants of $h + \iota(h)$ for admissible elements $h \in \Z[\Gamma]$.

%%%%%%%%%%%%%%%%%%%%%%
\subsection{Study of $h + \iota(h)$}\label{ss:power_series_2}
%%%%%%%%%%%%%%%%%%%%%%

Let $A$ be again any commutative ring.
In this subsection, we study $h + \iota(h)$ for $h \in A[[T]]$.

\begin{defn}\label{defn:Phi}
For any $h \in A[[T]]$, since $h + \iota(h) \in A[[T]]^{\lrang{\iota}}$, we may apply Lemma \ref{lem:inv_exp} to construct a unique $g \in A[[S]]$ such that
\[
g(T + \iota(T)) = h + \iota(h).
\]
By setting $\Phi(h) = g$, we define an $A$-homomorphism
\[
\Phi: A[[T]] \to A[[S]].
\]
\end{defn}

Let us obtain formulas for $\Phi(T^k)$.

\begin{prop}\label{prop:g_k_2}
Set $g_k(S) = \Phi(T^k)$ for $k = 0, 1, 2, \dots$.
Then we have $g_0(S) = 2$, $g_1(S) = S$, and
\[
g_k(S) = S(g_{k-1}(S) + g_{k-2}(S))
\]
for $k \geq 2$.
For any $k \geq 1$, the power series $g_k(S)$ is indeed a monic polynomial of degree $k$ with $\ord_S(g_k(S)) \geq k/2$.
\end{prop}

\begin{proof}
By definition, $g_k(S)$ is characterized by
\[
g_k(T + \iota(T)) = T^k + \iota(T)^k.
\]
Then the formulas for $g_0$ and $g_1$ follow easily.
For each $k \geq 2$, we have
\[
(T + \iota(T)) (T^{k-1} + \iota(T)^{k-1})
= (T^k + \iota(T)^k) + T \cdot \iota(T) (T^{k-2} + \iota(T)^{k-2}).
\]
Noting that $T \cdot \iota(T) = -(T + \iota(T))$, we obtain
\[
g_1(S) g_{k-1}(S) = g_k(S) - S g_{k-2}(S),
\]
which implies the formula.
The final assertion follows inductively.
\end{proof}

Here is the list of $\Phi(T^k)$ for $k = 0, 1, 2, \dots, 7$.
\begin{align*}
\Phi(T^0) &= 2\\
\Phi(T^1) &= S\\
\Phi(T^2) &= S^2 + 2S \\
\Phi(T^3) &= S^3 + 3S^2 \\
\Phi(T^4) &= S^4 + 4S^3 + 2S^2 \\
\Phi(T^5) &= S^5 + 5S^4 + 5S^3 \\
\Phi(T^6) &= S^6 + 6S^5 + 9S^4 + 2S^3 \\
\Phi(T^7) &= S^7 + 7S^6 + 14S^5 + 7S^4
\end{align*}

By using Proposition \ref{prop:g_k_2}, we easily deduce the following.

\begin{prop}\label{prop:Phi_surj}
The restriction of the homomorphism $\Phi: A[[T]] \to A[[S]]$ gives an isomorphism $\Phi: TA[T] \overset{\simeq}{\to} S A[S]$.
\end{prop}

Note that, since $\Phi(1) = 2$, the restriction of $\Phi$ gives the duplication map $A \to A$ on the constant terms, which is not surjective or injective in general.

It will be also convenient to compute $\Phi((1+T)^k-1)$.
There is a natural relation between $\Phi(T^k)$ and $\Phi((1 + T)^k-1)$, but we provide an individual proposition.

\begin{prop}\label{prop:g_k}
Set $g_k(S) := \Phi((1+T)^k-1)$ for $k = 0, 1, 2, \dots$.
Then we have $g_0(S) = 0$, $g_1(S) = S$, and
\[
g_k(S) = (S + 2)g_{k-1}(S) - g_{k-2}(S) + 2S
\]
for $k \geq 2$.
For any $k \geq 1$, the power series $g_k(S)$ is indeed a monic polynomial of degree $k$ with $\ord_S(g_k(S)) \geq 1$.
\end{prop}

\begin{proof}
This is proved similarly to Proposition \ref{prop:g_k_2}.
\end{proof}

Here is the list of $\Phi((1+T)^k-1)$ for $k = 0, 1, 2, \dots, 7$.
\begin{align*}
\Phi((1+T)^0-1) &= 0\\
\Phi((1+T)^1-1) &= S\\
\Phi((1+T)^2-1) &= S^2 + 4S\\
\Phi((1+T)^3-1) &= S^3 + 6S^2 + 9S\\
\Phi((1+T)^4-1) &= S^4 + 8S^3 + 20 S^2 + 16S\\
\Phi((1+T)^5-1) &= S^5+10S^4+35S^3+50S^2+25S\\
\Phi((1+T)^6-1) &= S^6+12S^5+54S^4+112S^3+105S^2+36S\\
\Phi((1+T)^7-1) &= S^7+14S^6+77S^5+210S^4+294S^3+196S^2+49S
\end{align*}

\begin{rem}
These polynomials (up to sign) already appear in Valli\`{e}res \cite[Lemma 5.5]{Val21}.
Moreover, McGown--Valli\`{e}res \cite[Section 2]{MV23} observe that they are essentially the Chebyshev polynomials.
In fact, in our notation, we have
\[
\Phi((1+T)^k-1) = 2 T_k(S/2 + 1) - 2,
\]
where $T_k$ denotes the $k$-th Chebyshev polynomial.
This identity follows from comparing the definition of $\Phi((1+T)^k-1)$ with $T_k((\zeta + \zeta^{-1})/2) = (\zeta^k + \zeta^{-k})/2$ for any root of unity $\zeta$.
\end{rem}

%%%%%%%%%%%%%%%%%%%%%%
\subsection{Construction}\label{ss:constr}
%%%%%%%%%%%%%%%%%%%%%%

Now we are ready to prove the main theorem.

\begin{proof}[Proof of Theorem \ref{thm:main}]
We return to the setting in Subsection \ref{ss:basic_bouquet}.
For each voltage assignment $\alpha$, we have $h_{\alpha} + \iota(h_{\alpha}) = \Phi(h_{\alpha})(T + \iota(T))$ by Definition \ref{defn:Phi}.
Then using Propositions \ref{prop:inv_g} and \ref{prop:bou_compu}, we obtain
\[
\lambda(X_{\infty}/X) = \lambda(h_{\alpha} + \iota(h_{\alpha})) - 1 = 2 \lambda_S(\Phi(h_{\alpha}))-1
\]
and
\[
\mu(X_{\infty}/X) = \mu(h_{\alpha} + \iota(h_{\alpha})) = \mu_S(\Phi(h_{\alpha})).
\]

Let $\lambda' \geq 1$ and $\mu \geq 0$ be given integers.
We shall construct a $\Z_p$-covering $X_{\infty}/X$ with $X$ a bouquet such that $\lambda(X_{\infty}/X) = 2 \lambda' - 1$ and $\mu(X_{\infty}/X) = \mu$.
By the discussion above, these are equivalent to $\lambda_S(\Phi(h_{\alpha})) = \lambda'$ and $\mu_S(\Phi(h_{\alpha})) = \mu$.

Take any polynomial $g(S) \in S\Z[S]$ satisfying $\lambda_S(g(S)) = \lambda'$ and $\mu_S(g(S)) = \mu$; for instance, $g(S) = p^{\mu} S^{\lambda'}$ suffices.
Then by Proposition \ref{prop:Phi_surj}, we find a unique polynomial $h \in T \Z[T]$ such that $\Phi(h) = g(S)$.

It remains to modify $h$ into an admissible element in the sense of Definition \ref{defn:adm}.
Condition (a) is already satisfied.
We can uniquely write in a finite sum
\[
h = \sum_{k \geq 1} c_k (1 - (1+T)^k)
= \sum_{k \geq 1} c_k (1 - \gamma^k)
\]
with $c_k \in \Z$.
We may freely change $c_k$ modulo $p^{\mu+1} \Z$ since this does not affect the first condition for $g(S)$.
Therefore, we can assume that $c_k \geq 0$, so condition (b) is satisfied.
By adding $p^{\mu+1}(1 - \gamma)$ if necessary, condition (c) can be satisfied.
This completes the proof.
\end{proof}

Let us describe the construction through concrete examples.
Recall that we listed $\Phi((1+T)^k-1)$ following Proposition \ref{prop:g_k}.
We need to compute the inverse of the corresponding matrix representation:
\[
\begin{pmatrix}
1 & 4 & 9 & 16 & 25 & 36 & 49 & \cdots \\
  & 1 & 6 & 20 & 50 & 105 & 196 & \cdots \\
  &   & 1 & 8 & 35 & 112 & 294 & \cdots \\
  &   &   & 1 & 10 & 54 & 210 & \cdots \\
  &   &   &   & 1 & 12 & 77 & \cdots \\
  &   &   &   &   & 1 & 14 & \cdots \\
  &   &   &   &   &   & 1 & \cdots \\
  &   &   &   &   &   &   & \ddots
\end{pmatrix}^{-1}
=
\begin{pmatrix}
1 & -4 & 15 & -56 & 210 & -792 & 3003 & \cdots \\
  & 1 & -6 & 28 & -120 & 495 & -2002 & \cdots \\
  &   & 1 & -8 & 45 & -220 & 1001 & \cdots \\
  &   &   & 1 & -10 & 66 & -364 & \cdots \\
  &   &   &   & 1 & -12 & 91 & \cdots \\
  &   &   &   &   & 1 & -14 & \cdots \\
  &   &   &   &   &   & 1 & \cdots \\
  &   &   &   &   &   &   & \ddots
\end{pmatrix}.
\]
This should be read as the following table.

\begin{table}[ht]
\centering
\begin{tabular}{c | c}
$\Phi(h)$ & $h$ \\
\hline
$S$ & $(\gamma - 1)$ \\
$S^2$ & $(\gamma^2 - 1) - 4(\gamma - 1)$ \\
$S^3$ & $(\gamma^3 - 1) - 6(\gamma^2 - 1) + 15(\gamma - 1)$ \\
$\vdots$ & $\vdots$
\end{tabular}
\end{table}

We first assume $\mu = 0$ and show how to construct examples for $\lambda' = 1, 2$ and $3$ (so $\lambda(X_{\infty}/X) = 1, 3$ and $5$).

\begin{itemize}
\item
Let $\lambda' = 1$.
We take $g(S) = S$.
Then $g(S) = \Phi(h)$ with 
\[
h = \gamma - 1 \equiv (p-1)(1 - \gamma) \mod p.
\]
Therefore, we may take $t = p-1$ and $\alpha(e_1) = \dots = \alpha(e_{p-1}) = \gamma$.
Alternatively, we may first take $g(S) = -S$, in which case $h = 1 - \gamma$, so $t = 1$ and $\alpha(e_1) = \gamma$ suffices (see Figure \ref{fig:case_p3l12}).
This kind of refinements are often possible.

%p=3, lambda = 1, mu = 0
\begin{figure}[H]
  \centering
%\Large
\[
\xymatrix{
%0-th
\begin{tikzpicture}[baseline]
\node[circle,fill,inner sep=1pt] (a) {};
\foreach \len in {30}
\path (a) edge[loop above, looseness=\len, out=60, in=120] node[above] {\normalsize $\gamma$}(a);
\end{tikzpicture}
&
%1-st
\begin{tikzpicture}[baseline]
	\node[minimum size=4em,regular polygon,regular polygon sides=3] (a) {};
	\foreach \x in {1, 2, ..., 3}
	\fill (a.corner \x) circle[radius=1pt];
	\foreach \x/\y in {1/2, 2/3, 3/1}
	\foreach \bend in {0}
		\path (a.corner \x) edge [bend right=\bend] (a.corner \y);
\end{tikzpicture}
\ar[l] &
%2-nd
\begin{tikzpicture}[baseline]
	\node[minimum size=4em,regular polygon,regular polygon sides=9] (a) {};
	\foreach \x in {1, 2, ..., 9}
	\fill (a.corner \x) circle[radius=1pt];
	\foreach \x/\y in {1/2, 2/3, 3/4, 4/5, 5/6, 6/7, 7/8, 8/9, 9/1}
	\foreach \bend in {0}
		\path (a.corner \x) edge [bend right=\bend]  (a.corner \y);
\end{tikzpicture}
\ar[l] & \cdots \ar[l]}
\]
  \caption{($p = 3$) $\lambda = 1$, $\mu = 0$}
    \label{fig:case_p3l12}
\end{figure}

\item
Let $\lambda' = 2$.
We take $g(S) = S^2$.
Then 
\[
h = (\gamma^2 - 1) - 4(\gamma-1)
\equiv (p-1) (1 - \gamma^2) +  4 (1 - \gamma) \mod p.
\]
So we may take $t = p+3$ and
\[
\alpha(e_1) = \dots = \alpha(e_{p-1}) = \gamma^2, \alpha(e_p) = \dots = \alpha(e_{p+3}) = \gamma.
\]
If $p = 2$, since $h \equiv (1-\gamma^2) + 2(1-\gamma)$, we may take $t = 3$ and $\alpha(e_1) = \gamma^2, \alpha(e_2) = \alpha(e_3) = \gamma$ (see Figure \ref{fig:case_p2l3}).
If $p = 3$, since $h \equiv 2 (1 - \gamma^2) +  (1 - \gamma)$, we may take $t = 3$ and $\alpha(e_1) = \alpha(e_2) = \gamma^2, \alpha(e_3) = \gamma$ (see Figure \ref{fig:case_p3l3}).

%p=2, lambda = 3, mu=0
\begin{figure}[H]
  \centering
%\Large
\[
\xymatrix{
%0-th
\begin{tikzpicture}[baseline]
\node[circle,fill,inner sep=1pt] (a) {};
\foreach \len in {15,30}
\path (a) edge[loop above, looseness=\len, out=60, in=120] (a);
\path (a) edge[loop above, looseness=45, out=60, in=120] node[above] {\normalsize $\gamma^2, \gamma, \gamma$} (a);
\end{tikzpicture}
&
%1-st
\begin{tikzpicture}[baseline]
\node[circle,fill,inner sep=1pt] (a) at (0,0.5) {};
\node[circle,fill,inner sep=1pt] (b) at (0,-0.5) {};
\foreach \bend in {15, 5, -5,-15}
\path (a) edge[draw,bend right=\bend] (b);
\path (a) edge[loop left, looseness=15, out=120, in=60] (a);
\path (b) edge[loop left, looseness=15, out=-120, in=-60] (b);
\end{tikzpicture}
\ar[l] &
%2-nd
\begin{tikzpicture}[baseline]
	\node[minimum size=4em,regular polygon,regular polygon sides=4] (a) {};
	\foreach \x in {1, 2, ..., 4}
	\fill (a.corner \x) circle[radius=1pt];
	\foreach \x/\y in {1/2, 2/3, 3/4, 4/1}
	\foreach \bend in {10, -10}
		\path (a.corner \x) edge [bend right=\bend] (a.corner \y);
	\foreach \x/\y in {1/3, 2/4, 3/1, 4/2}
	\foreach \bend in {10}
		\path (a.corner \x) edge [bend right=\bend] (a.corner \y);
\end{tikzpicture}
\ar[l] &
%3-rd
\begin{tikzpicture}[baseline]
	\node[minimum size=4em,regular polygon,regular polygon sides=8] (a) {};
	\foreach \x in {1, 2, ..., 8}
	\fill (a.corner \x) circle[radius=1pt];
	\foreach \x/\y in {1/2, 2/3, 3/4, 4/5, 5/6, 6/7, 7/8, 8/1}
	\foreach \bend in {10, -10}
		\path (a.corner \x) edge [bend right=\bend] (a.corner \y);
	\foreach \x/\y in {1/3, 2/4, 3/5, 4/6, 5/7, 6/8, 7/1, 8/2}
	\foreach \bend in {0}
		\path (a.corner \x) edge [bend right=\bend] (a.corner \y);
\end{tikzpicture}
\ar[l] & \cdots \ar[l]}
\]
  \caption{($p = 2$) $\lambda = 3$, $\mu = 0$}
  \label{fig:case_p2l3}
\end{figure}

%p=3, lambda = 3, mu=0
\begin{figure}[H]
  \centering
%\Large
\[
\xymatrix{
%0-th
\begin{tikzpicture}[baseline]
\node[circle,fill,inner sep=1pt] (a) {};
\foreach \len in {15,30}
\path (a) edge[loop above, looseness=\len, out=60, in=120] (a);
\path (a) edge[loop above, looseness=45, out=60, in=120] node[above] {\normalsize $\gamma^2, \gamma^2, \gamma$} (a);
\end{tikzpicture}
&
%1-st
\begin{tikzpicture}[baseline]
	\node[minimum size=4em,regular polygon,regular polygon sides=3] (a) {};
	\foreach \x in {1, 2, ..., 3}
	\fill (a.corner \x) circle[radius=1pt];
	\foreach \x/\y in {1/2, 2/3, 3/1}
	\foreach \bend in {15, 0, -15}
		\path (a.corner \x) edge [bend right=\bend] (a.corner \y);
\end{tikzpicture}
\ar[l] &
%2-nd
\begin{tikzpicture}[baseline]
	\node[minimum size=4em,regular polygon,regular polygon sides=9] (a) {};
	\foreach \x in {1, 2, ..., 9}
	\fill (a.corner \x) circle[radius=1pt];
	\foreach \x/\y in {1/2, 2/3, 3/4, 4/5, 5/6, 6/7, 7/8, 8/9, 9/1}
	\foreach \bend in {0}
		\path (a.corner \x) edge [bend right=\bend] (a.corner \y);
	\foreach \x/\y in {1/3, 2/4, 3/5, 4/6, 5/7, 6/8, 7/9, 8/1, 9/2}
	\foreach \bend in {10, -10}
		\path (a.corner \x) edge [bend right=\bend] (a.corner \y);
\end{tikzpicture}
\ar[l] & \cdots \ar[l]}
\]
  \caption{($p = 3$) $\lambda = 3$, $\mu = 0$}
    \label{fig:case_p3l3}
\end{figure}

\item
Let $\lambda' = 3$.
We take $g(S) = S^3$.
Then 
\[
h = (\gamma^3-1) - 6(\gamma^2-1) +15 (\gamma -1).
\]
We have to modify this to make it admissible; a uniform choice is
\[
h \equiv (p-1) (1 - \gamma^3) + 6(1 - \gamma^2) + 15 (p-1) (1 - \gamma) \mod p.
\]
For specific values of $p$, we can reduce the number of edges.
For example, if $p = 2$, since $h \equiv (1 - \gamma^3) + (1 - \gamma)$, we may take $t = 2$ and $\alpha(e_1) = \gamma^3, \alpha(e_2) = \gamma$ (see Figure \ref{fig:case_p2l5}).
If $p = 3$, since $h \equiv (1 - \gamma^3) + 3 (1 - \gamma)$, we may take $t = 4$ and $\alpha(e_1) = \gamma^3, \alpha(e_2) = \alpha(e_3) = \alpha(e_4) = \gamma$ (see Figure \ref{fig:case_p3l5}).

%p=2, lambda = 3, mu=0
\begin{figure}[H]
  \centering
%\Large
\[
\xymatrix{
%0-th
\begin{tikzpicture}[baseline]
\node[circle,fill,inner sep=1pt] (a) {};
\foreach \len in {15}
\path (a) edge[loop above, looseness=\len, out=60, in=120] (a);
\path (a) edge[loop above, looseness=30, out=60, in=120] node[above] {\normalsize $\gamma^3, \gamma$} (a);
\end{tikzpicture}
&
%1-st
\begin{tikzpicture}[baseline]
\node[circle,fill,inner sep=1pt] (a) at (0,0.5) {};
\node[circle,fill,inner sep=1pt] (b) at (0,-0.5) {};
\foreach \bend in {15, 5, -5,-15}
\path (a) edge[draw,bend right=\bend] (b);
\end{tikzpicture}
\ar[l] &
%2-nd
\begin{tikzpicture}[baseline]
	\node[minimum size=4em,regular polygon,regular polygon sides=4] (a) {};
	\foreach \x in {1, 2, ..., 4}
	\fill (a.corner \x) circle[radius=1pt];
	\foreach \x/\y in {1/2, 2/3, 3/4, 4/1}
	\foreach \bend in {10, -10}
		\path (a.corner \x) edge [bend right=\bend] (a.corner \y);
\end{tikzpicture}
\ar[l] &
%3-rd
\begin{tikzpicture}[baseline]
	\node[minimum size=4em,regular polygon,regular polygon sides=8] (a) {};
	\foreach \x in {1, 2, ..., 8}
	\fill (a.corner \x) circle[radius=1pt];
	\foreach \x/\y in {1/2, 2/3, 3/4, 4/5, 5/6, 6/7, 7/8, 8/1}
	\foreach \bend in {0}
		\path (a.corner \x) edge [bend right=\bend] (a.corner \y);
	\foreach \x/\y in {1/4, 2/5, 3/6, 4/7, 5/8, 6/1, 7/2, 8/3}
	\foreach \bend in {0}
		\path (a.corner \x) edge [bend right=\bend] (a.corner \y);
\end{tikzpicture}
\ar[l] & \cdots \ar[l]}
\]
  \caption{($p = 2$) $\lambda = 5$, $\mu = 0$}
  \label{fig:case_p2l5}
\end{figure}

%p=3, lambda = 5, mu=0
\begin{figure}[H]
  \centering
%\Large
\[
\xymatrix{
%0-th
\begin{tikzpicture}[baseline]
\node[circle,fill,inner sep=1pt] (a) {};
\foreach \len in {15,30,45}
\path (a) edge[loop above, looseness=\len, out=60, in=120] (a);
\path (a) edge[loop above, looseness=60, out=60, in=120] node[above] {\normalsize $\gamma^3, \gamma, \gamma, \gamma$} (a);
\end{tikzpicture}
&
%1-st
\begin{tikzpicture}[baseline]
	\node[minimum size=4em,regular polygon,regular polygon sides=3] (a) {};
	\node[circle,fill,inner sep=1pt] (a1) at (a.corner 1) {};
	\node[circle,fill,inner sep=1pt] (a2) at (a.corner 2) {};
	\node[circle,fill,inner sep=1pt] (a3) at (a.corner 3) {};
\path (a1) edge[loop, looseness=30, out=60, in=120] (a1);
\path (a2) edge[loop, looseness=30, out=180, in=240] (a2);
\path (a3) edge[loop, looseness=30, out=300, in=0] (a3);
\path (0,0) edge[loop, looseness=10, out=45, in=135] (0,0);
	\foreach \x in {1, 2, ..., 3}
	\fill (a.corner \x) circle[radius=1pt];
	\foreach \x/\y in {1/2, 2/3, 3/1}
	\foreach \bend in {15, 0, -15}
		\path (a.corner \x) edge [bend right=\bend] (a.corner \y);
\end{tikzpicture}
\ar[l] &
%2-nd
\begin{tikzpicture}[baseline]
	\node[minimum size=4em,regular polygon,regular polygon sides=9] (a) {};
	\foreach \x in {1, 2, ..., 9}
	\fill (a.corner \x) circle[radius=1pt];
	\foreach \x/\y in {1/2, 2/3, 3/4, 4/5, 5/6, 6/7, 7/8, 8/9, 9/1}
	\foreach \bend in {15,0,-15}
		\path (a.corner \x) edge [bend right=\bend] (a.corner \y);
	\foreach \x/\y in {1/4, 2/5, 3/6, 4/7, 5/8, 6/9, 7/1, 8/2, 9/3}
	\foreach \bend in {0}
		\path (a.corner \x) edge [bend right=\bend] (a.corner \y);
\end{tikzpicture}
\ar[l] & \cdots \ar[l]}
\]
  \caption{($p = 3$) $\lambda = 5$, $\mu = 0$}
    \label{fig:case_p3l5}
\end{figure}
\end{itemize}

Table \ref{tab:unram} lists the values of $\ord_p(\kappa_{X_n})$ for these examples, computed using SageMath.

\begin{table}[htbp]
\centering
\begin{tabular}{c @{\hspace{2em}} cccccccc}
\toprule
$n$ & \swid{$0$} & \swid{$1$} & \swid{$2$} & \swid{$3$} & \swid{$4$} & \swid{$5$} & \swid{$6$} & $\cdots$ \\
\midrule
Figure \ref{fig:case_p3l12} & $0$ & $1$ & $2$ & $3$ & $4$ & $5$ & $6$ & $\cdots$ \\
Figure \ref{fig:case_p2l3} & $0$ & $2$ & $7$ & $10$ & $13$ & $16$ & $19$ & $\cdots$ \\
Figure \ref{fig:case_p3l3} & $0$ & $3$ & $6$ & $9$ & $12$ & $15$ & $18$ & $\cdots$ \\
Figure \ref{fig:case_p2l5} & $0$ & $2$ & $5$ & $12$ & $17$ & $22$ & $27$ & $\cdots$ \\
Figure \ref{fig:case_p3l5} & $0$ & $3$ & $8$ & $13$ & $18$ & $23$ & $28$ & $\cdots$ \\
\bottomrule
\end{tabular}
\caption{The values of $\ord_p(\kappa_{X_n})$}
\label{tab:unram}
\end{table}

When $\mu \geq 1$, we only need to multiply these polynomials $g$ and $h$ by $p^{\mu}$.
Geometrically, this means multiplying the number of edges by $p^{\mu}$.

\begin{rem}
We briefly discuss graphs $X$ that are not bouquets.
We fix $n = \# V_X$, the number of vertices of $X$.
Then naturally we should deal with matrices $h \in M_n(\Z[\Gamma])$ satisfying a certain admissibility condition, and study the $\lambda$- and $\mu$-invariants of the determinant $\det(h + \iota(h)) \in \Z[\Gamma]$.
The difficulty is that the map
\[
M_n(\Z[\Gamma]) \to \Z[\Gamma], 
\quad
h \mapsto \det(h + \iota(h))
\]
is not a $\Z$-homomorphism, so its analysis becomes more complicated.
\end{rem}

%%%%%%%%%%%%%%%%%%%%%%
\section{Remarks}\label{sec:remark}
%%%%%%%%%%%%%%%%%%%%%%

We still focus on unramified $\Z_p$-coverings of bouquets $X$.

%%%%%%%%%%%%%%%%%%%%%%
\subsection{Probability}\label{ss:probab}
%%%%%%%%%%%%%%%%%%%%%%

Let us discuss the probabilities of prescribed pairs $(\lambda, \mu)$ as the $\Z_p$-covering $X_{\infty}/X$ varies.
For this purpose, we first fix the meaning of probability.
Our probability space is different from and easier than that of \cite{DLRV24}.

As discussed in Section \ref{sec:bouquets}, we are studying the pair $(2 \lambda_S(\Phi(h)) - 1, \mu_S(\Phi(h)))$, where $h \in \Z[\Gamma]$ is admissible in the sense of Definition \ref{defn:adm}.
However, to discuss the probability, we ignore conditions (b)(c) for the admissibility, and moreover allow $h$ to be in $\Z_p[[\Gamma]]$ instead of $\Z[\Gamma]$.
Namely, for each element $h$ of
\[
\Z_p[[\Gamma]]^{\aug = 0}
= \{h \in \Z_p[[\Gamma]] \mid \aug(h) = 0\},
\]
we consider as if there is an associated $\Z_p$-covering $X_{\infty}/X$, and its $\lambda$- and $\mu$-invariants are defined by the previous formulas
\[
\lambda(X_{\infty}/X) = 2 \lambda_S(\Phi(h))-1,
\quad
\mu(X_{\infty}/X) = \mu_S(\Phi(h)).
\]
Since $\Z_p[[\Gamma]]^{\aug = 0} \simeq T \Z_p[[T]]$ is a compact additive group, we can use the Haar measure to define the probability.

Let us justify this idea.
Firstly, the set of admissible elements is a dense subset of $\Z_p[[\Gamma]]^{\aug = 0}$.
Secondly, as discussed in Subsection \ref{ss:constr}, the Iwasawa invariants are stable under $p$-adically small change of the voltage assignments.
(Alternatively and more naturally, one can instead consider edge-weighted graphs with weights in $\Z_p$ in the sense of Adachi--Mizuno--Tateno \cite{AMT25}.)
Therefore, taking the topological closure sounds reasonable.

Now we prove the following.

\begin{prop}
We consider $X_{\infty}/X$ associated to each $h \in \Z_p[[\Gamma]]^{\aug=0}$.
Then we have $\mu(X_{\infty}/X) = 0$ almost everywhere and the probability of $\lambda(X_{\infty}/X) = 2\lambda'-1$ is $\cfrac{1}{p^{\lambda'-1}} \left(1 - \cfrac{1}{p} \right)$.
\end{prop}

\begin{proof}
We have $\lambda_S(\Phi(h)) \geq \lambda'$ or $\mu_S(\Phi(h)) \geq 1$ if and only if $\Phi(h) \in (p, S^{\lambda'}) \Z_p[[S]]$.
This is then equivalent to that $h$ is in the kernel of the homomorphism
\[
T\Z_p[[T]] \overset{\Phi}{\to} S\Z_p[[S]] \twoheadrightarrow S\F_p[[S]]/S^{\lambda'} \F_p[[S]],
\]
where the final map is the natural projection.
The probability of this event is $1/p^{\lambda'-1}$, since the cardinality of $S\F_p[[S]]/S^{\lambda'} \F_p[[S]]$ is $p^{\lambda'-1}$, and the displayed composite homomorphism is surjective by Proposition \ref{prop:Phi_surj}.
Therefore, the probability of $\lambda_S(\Phi(h)) = \lambda'$ and $\mu_S(\Phi(h)) = 0$ is
\[
\cfrac{1}{p^{\lambda'-1}} - \cfrac{1}{p^{\lambda'}}
= \cfrac{1}{p^{\lambda'-1}} \left(1 - \cfrac{1}{p} \right).
\]
Since $\displaystyle \sum_{\lambda'=1}^{\infty} \cfrac{1}{p^{\lambda'-1}} \left(1 - \cfrac{1}{p} \right) = 1$, we have $\mu = 0$ almost everywhere.
\end{proof}

As demonstrated in Subsection \ref{ss:constr}, we have an algorithm to determine the set of $h \in \Z_p[[\Gamma]]^{\aug=0}$ for which we have $\lambda(X_{\infty}/X) = 2\lambda'-1$ and $\mu(X_{\infty}/X) = 0$.
In this proposition, such a computation is unnecessary, since the surjectivity of $\Phi$ alone suffices.

%%%%%%%%%%%%%%%%%%%%%%
\subsection{Non-$p$-components}\label{ss:non-p}
%%%%%%%%%%%%%%%%%%%%%%

Let $X_{\infty}/X$ be an unramified $\Z_p$-covering of connected graphs.
Let $l \neq p$ be a prime number.
Then Lei--Valli\`{e}res \cite{LV23} showed that there exist a non-negative integer $\mu_l = \mu_l(X_{\infty}/X)$ and an integer $\nu_l$ such that
\[
\ord_l(\kappa_{X_n}) = \mu_l l^n + \nu_l
\]
holds for $n \gg 0$.
Note that $\lambda_l$ does not appear.
For clarification, we write $\lambda_p = \lambda_p(X_{\infty}/X)$ and $\mu_p = \mu_p(X_{\infty}/X)$ for those describing the growth of $\ord_p(\kappa_{X_n})$.

This $\mu_l$ is computed in the same way as $\mu_p$, as follows.
As in Subsection \ref{ss:Iw_graph}, we suppose that $X_{\infty}/X$ is obtained as the derived graphs of a voltage assignment $\alpha: \bE_X \to \Gamma$.
Then we have the element $Z_{\alpha} \in \Z[\Gamma]$ and $\mu_l$ is the largest integer such that
\[
Z_{\alpha} \in l^{\mu_l} \Z[\Gamma].
\]
Note that this is equivalent to $Z_{\alpha} \in l^{\mu_l} \Z_l[\Gamma]$ and that Lemma \ref{lem:divis1} (for $A = \Z$ or $\Z_l$) enables us to reformulate these conditions.

For each $X_{\infty}/X$, we have $\mu_l(X_{\infty}/X) = 0$ for almost all $l$.
This is simply because any integer divisible by infinitely many prime numbers must be $0$.
Let us show that this is the only additional restriction, improving Theorem \ref{thm:main}.

\begin{thm}\label{thm:main_l}
Given an odd integer $\lambda_p \geq 1$ and an integer $\mu_l \geq 0$ for each prime number $l$ (including $p$), as long as $\mu_l = 0$ holds for almost all $l$, there exists an unramified $\Z_p$-covering $X_{\infty}/X$ of connected graphs such that $\lambda_p(X_{\infty}/X) = \lambda_p$ and $\mu_l(X_{\infty}/X) = \mu_l$ holds for each prime $l$.
Indeed, we construct such a covering with $X$ a bouquet.
\end{thm}

\begin{proof}
By applying Theorem \ref{thm:main} to the given $\lambda_p$ (without conditions on $\mu_p$), we obtain an admissible element $h \in \Z[\Gamma]$ such that $\lambda_S(\Phi(h)) = (\lambda_p+1)/2$.
Indeed, in the proof, we constructed such an $h$ in $\Z[T] \subset \Z[\Gamma]$.
Recall that by Proposition \ref{prop:Phi_surj}
\[
\Phi: T \Z[T] \to S \Z[S]
\]
is an isomorphism, so Lemma \ref{lem:divis1} implies that the $\mu_l$-invariants are preserved under the map $\Phi$.
It is now enough to multiply $h$ by the positive rational number to adjust the $\mu_l$-invariants, noting that admissibility is preserved under this process.
\end{proof}

%%%%%%%%%%%%%%%%%%%%%%
\section{Proof of Theorem \ref{thm:main_ram}}\label{sec:ram}
%%%%%%%%%%%%%%%%%%%%%%

In this section, we study the ramified case.

%%%%%%%%%%%%%%%%%%%%%%
\subsection{The analogue of Iwasawa's class number formula}\label{ss:ram_basic}
%%%%%%%%%%%%%%%%%%%%%%

This subsection is a generalization of Subsection \ref{ss:Iw_graph}.
One may refer to Gambheera--Valli\`{e}res \cite{GV24} or the author's article \cite{Kata_31} for details.

%%%%%%%%%%%%%%%%%%%%%%
\subsubsection{Ramified $\Z_p$-coverings} 
%%%%%%%%%%%%%%%%%%%%%%

Let $X$ be a graph, $\Gamma$ a profinite group that is isomorphic to $\Z_p$, and $\alpha: \bE_X \to \Gamma$ a voltage assignment.
We moreover consider a family $\cI = (I_v \subset \Gamma)_{v \in V_X}$ of subgroups of $\Gamma$ indexed by the vertices.
Then we can construct the (ramified) tower of graphs
\[
X = X_0 \leftarrow X_1 \leftarrow X_2 \leftarrow \cdots
\]
as follows.
For each $n$, let $I_{v, n}$ be the image of $I_v$ to $\Gamma_n$.
We define the derived graph $X_n$ as
\[
V_{X_n} = \coprod_{v \in V_X} (\Gamma_n/I_{v, n} \times \{v\}),
\quad
\bE_{X_n} = \Gamma_n \times \bE_X
\]
and
\[
\ol{(\gamma, e)} = (\gamma \cdot \alpha(e), \ol{e}),
\quad
s((\gamma, e)) = (\gamma I_{s(e), n}, s(e)),
\quad
t((\gamma, e)) = (\gamma \cdot \alpha(e) I_{t(e), n}, t(e))
\]
for each $(\gamma, e) \in \bE_{X_n}$.
For $n' \geq n$, the morphism $X_{n'} \to X_n$ in the tower is induced from the projections $\Gamma_{n'}/I_{v, n'} \to \Gamma_n/I_{v, n}$ and $\Gamma_{n'} \to \Gamma_n$.
For each $n$, the group $\Gamma_n$ naturally acts on $X_n$ over $X$.

As long as $X_n$ are all connected, the tower above is a $\Z_p$-covering of connected graphs.
Note that the connectedness is guaranteed if the unramified derived graphs are already connected, or if $X$ is connected and has a totally ramified vertex $v$ (i.e., $I_v = \Gamma$).
Conversely, every $\Z_p$-covering of connected graphs can be constructed in this manner, up to isomorphism.

%%%%%%%%%%%%%%%%%%%%%%
\subsubsection{The formula}
%%%%%%%%%%%%%%%%%%%%%%

We state the analogue of Iwasawa's class number formula with explicit descriptions of the $\lambda$- and $\mu$-invariants.
The result is essentially the same as \cite[Theorems A and B]{GV24}.
We follow the argument in \cite[Proof of Theorem 4.3]{Kata_31}.
Associated to the family $\cI$, we define a subset of $V_X$ by
\[
V_{X, \cI} = \{v \in V_X \mid \text{$I_v$ is trivial}\},
\]
which is denoted by $V_X^0$ in \cite{Kata_31}.
Recall that the Laplacian operator $\cL_{\alpha}: \bigoplus_{v \in V_X} \Z[\Gamma] v \to \bigoplus_{v \in V_X} \Z[\Gamma] v$ is defined in Subsection \ref{ss:Iw_graph}.
Let
\[
\cL_{\alpha, \cI}: \bigoplus_{v \in V_{X, \cI}} \Z[\Gamma] v \to \bigoplus_{v \in V_{X, \cI}} \Z[\Gamma] v
\]
be the component (a submatrix) of $\cL_{\alpha}$, which is denoted by $\cL_{X, \Gamma, \cI}'$ in \cite{Kata_31}.
We define
\[
Z_{\alpha, \cI} = \det(\cL_{\alpha, \cI}) \in \Z[\Gamma].
\]

As a generalization of Theorem \ref{thm:ICNF_ref}, we have the following.

\begin{thm}\label{thm:ICNF_ref_ram}
Let $X$ be a connected graph.
Let $\alpha: \bE_X \to \Gamma$ be a voltage assignment.
Let $\cI = (I_v \subset \Gamma)_{v \in V_X}$ be a family of subgroups of $\Gamma$ indexed by the vertices.
Let $X_{\infty}/X$ be the ramified $\Z_p$-covering constructed by the derived graphs, where we assume that each $X_n$ is connected.
Then
\[
\ord_p(\kappa_{X_n}) = \lambda n + \mu p^n + \nu
\]
holds for $n \gg 0$ with
\[
\lambda = \lambda(Z_{\alpha, \cI}) - 1 + \sum_{v \in V_X \setminus V_{X, \cI}} [\Gamma: I_v],
\]
$\mu = \mu(Z_{\alpha, \cI})$, and some integer $\nu$.
\end{thm}

\begin{proof}
Let $\Pic X_n$ be the Picard group of $X_n$, and define $\Pic_{\Z_p} X_{\infty} = \varprojlim_n (\Z_p \otimes_{\Z} \Pic X_n)$.
Then \cite[Theorem A]{GV24} (see also \cite[Theorem 4.1]{Kata_31}) shows the formula with $\lambda = \lambda(\Pic_{\Z_p} X_{\infty}) - 1$ and $\mu = \mu(\Pic_{\Z_p} X_{\infty})$.
In \cite[Proof of Theorem 4.3]{Kata_31}, we derive an exact sequence
\[
0 \to \bigoplus_{v \in V_X \setminus V_{X, \cI}} \Z_p[\Gamma/I_v] \to \Pic_{\Z_p} X_{\infty} \to \Pic'_{\Z_p} X_{\infty} \to 0,
\]
where we define $\Pic'_{\Z_p} X_{\infty}$ as the cokernel of the endomorphism $\cL_{\alpha, \cI}$ on $\bigoplus_{v \in V_{X, \cI}} \Z_p[[\Gamma]] v$, so we have $\lambda(\Pic'_{\Z_p} X_{\infty}) = \lambda(Z_{\alpha, \cI})$ and $\mu(\Pic'_{\Z_p} X_{\infty}) = \mu(Z_{\alpha, \cI})$.
For each $v \in V_X \setminus V_{X, \cI}$, we also have $\lambda(\Z_p[\Gamma/I_v]) = [\Gamma: I_v]$ and $\mu(\Z_p[\Gamma/I_v]) = 0$.
Now the theorem follows by using the additivity of $\lambda$ and $\mu$ with respect to exact sequences.
\end{proof}

\begin{eg}
Suppose that $I_v$ is non-trivial for every $v \in V_X$, that is, $V_{X, \cI} = \emptyset$.
Then we have $Z_{\alpha, \cI} = 1$ (regardless of $\alpha$), so $\mu = 0$ and $\lambda = -1 + \sum_{v \in V_X} [\Gamma: I_v]$ hold.
In particular, by setting $I_v = \Gamma$ for every $v \in V_X$ and varying $\# V_X$, we can realize any $\lambda \geq 0$ with $\mu = 0$.
\end{eg}

%%%%%%%%%%%%%%%%%%%%%%
\subsubsection{The parity condition}
%%%%%%%%%%%%%%%%%%%%%%

The following is a generalization of Proposition \ref{prop:odd}.

\begin{prop}\label{prop:odd_2}
In Theorem \ref{thm:ICNF_ref_ram}, we have
\[
\lambda(X_{\infty}/X) \equiv - 1 + \sum_{v \in V_X \setminus V_{X, \cI}} [\Gamma: I_v] \mod 2.
\]
\end{prop}

\begin{proof}
By Proposition \ref{prop:iota_inv_Z}, $\iota(\cL_{\alpha})$ is equal to the transpose of $\cL_{\alpha}$, 
so the same holds for $\cL_{\alpha, \cI}$ and we obtain $\iota(Z_{\alpha, \cI}) = Z_{\alpha, \cI}$.
By Proposition \ref{prop:inv_g}, it follows that $\lambda(Z_{\alpha, \cI})$ is even.
This completes the proof.
\end{proof}

%%%%%%%%%%%%%%%%%%%%%%
\subsection{Construction}
%%%%%%%%%%%%%%%%%%%%%%

In order to construct any pair $(\lambda, \mu)$ with even $\lambda$, the discussion so far leads us to consider connected graphs $X$ with two vertices and $\Z_p$-coverings $X_{\infty}/X$ in which one vertex is unramified and the other is totally ramified.
Note that the existence of a totally ramified vertex implies that the $X_n$ are all connected.

The following corresponds to Proposition \ref{prop:bou_compu}.

\begin{prop}\label{prop:bou2_compu}
Let $X$ be a connected graph with two vertices, so $V_X = \{v, w\}$.
We define the family $\cI$ by letting $I_v$ be trivial and $I_w = \Gamma$.
Let $\alpha: \bE_X \to \Gamma$ be a voltage assignment, and let $X_{\infty}/X$ be the derived ramified $\Z_p$-covering.
Let
\[
\{e \in \bE_X \mid s(e) = t(e) = v\} = \{e_1, \ol{e_1}, \dots, e_t, \ol{e_t} \}
\]
and we set
\[
h_{\alpha} = \sum_{i=1}^t (1-\alpha(e_i)) \in \Z[\Gamma].
\]
We also set $c = \# \{e \in \bE_X \mid s(e) =v, t(e) = w\}$, which is positive since $X$ is connected.
Then we have
\[
\lambda(X_{\infty}/X) = \lambda(h_{\alpha} + \iota(h_{\alpha}) + c),
\quad
\mu(X_{\infty}/X) = \mu(h_{\alpha} + \iota(h_{\alpha}) + c).
\]
\end{prop}

\begin{proof}
This follows directly from Theorem \ref{thm:ICNF_ref_ram}.
We only have to note that 
\[
-1 + [\Gamma: I_w] = -1 + 1 = 0
\]
and $Z_{\alpha, \cI} = h_{\alpha} + \iota(h_{\alpha}) + c$.
\end{proof}

Now we are ready to prove the second main theorem.

\begin{proof}[Proof of Theorem \ref{thm:main_ram}]
We consider the situation in Proposition \ref{prop:bou2_compu}.
By using $\Phi$ in Definition \ref{defn:Phi}, we have
\[
h_{\alpha} + \iota(h_{\alpha}) + c = \Phi(h_{\alpha})(T + \iota(T)) + c,
\]
so we obtain
\[
\lambda(X_{\infty}/X) = 2 \lambda_S(\Phi(h_{\alpha}) + c),
\quad
\mu(X_{\infty}/X) = \mu_S(\Phi(h_{\alpha}) + c)
\]
by Proposition \ref{prop:inv_g}.

Let $\lambda' \geq 0$ and $\mu \geq 0$ be given integers.
We shall construct a $\Z_p$-covering $X_{\infty}/X$ such that $\lambda(X_{\infty}/X) = 2 \lambda'$ and $\mu(X_{\infty}/X) = \mu$.
By the discussion above, these are equivalent to $\lambda_S(\Phi(h_{\alpha}) + c) = \lambda'$ and $\mu_S(\Phi(h_{\alpha}) + c) = \mu$.

Take any polynomial $g(S) \in S\Z[S]$ and $c \in \Z$ satisfying $\lambda_S(g(S) + c) = \lambda'$ and $\mu_S(g(S) + c) = \mu$; for instance, we may take $g(S) = p^{\mu} S^{\lambda'}$ and $c = 0$ if $\lambda' \geq 1$, and $g(S) = 0$ and $c = p^{\mu}$ if $\lambda' = 0$.
Then by Proposition \ref{prop:Phi_surj}, we find a unique polynomial $h \in T \Z[T]$ such that $\Phi(h) = g(S)$.

It remains to modify $h$ and $c$ so that $h$ becomes admissible in the sense of Definition \ref{defn:adm} and $c$ becomes positive (note that condition (c) in admissibility can be ignored, since the connectedness holds without it).
This can be done in the same manner as in the proof of Theorem \ref{thm:main}, by modifying $h$ and $c$ modulo $p^{\mu+1}$.
This completes the proof.
\end{proof}

Let us describe the construction through concrete examples.
\begin{itemize}
\item
Let $\lambda' = 0$ and $\mu \geq 0$.
We take $g(S) = 0$ and $c = p^{\mu}$.
Then $g(S) = \Phi(h)$ with $h = 0$, which already satisfies the admissibility except for (c).
Therefore, we may take $t = 0$ and $c = p^{\mu}$ (see Figure \ref{fig:case_p3l1}).

%p=3, lambda=0, mu=1
\begin{figure}[H]
\centering
%\Large
\[
\xymatrix{
%0-th
\begin{tikzpicture}[baseline]
\node[circle,fill,inner sep=1pt] (a) at (0,0.5) {};
\node[circle,fill,inner sep=1pt] (b) at (0,-0.5) {};
\foreach \bend in {15, 0,-15}
\path (a) edge[draw,bend right=\bend] (b);
\end{tikzpicture}&
%1-st
\begin{tikzpicture}[baseline]
	\node[minimum size=4em,regular polygon,regular polygon sides=3] (a) {};
	\node[circle,fill,inner sep=1pt] (c) at (a.center) {};
	\foreach \x in {1, 2, ..., 3}
	\fill (a.corner \x) circle[radius=1pt];
	\foreach \x in {1, 2, ..., 3}
	\foreach \bend in {-15, 0, 15}
		\path (a.corner \x) edge [bend right=\bend] (c);	
\end{tikzpicture}
\ar[l] &
%2-nd
\begin{tikzpicture}[baseline]
	\node[minimum size=4em,regular polygon,regular polygon sides=9] (a) {};
	\node[circle,fill,inner sep=1pt] (c) at (a.center) {};
	\foreach \x in {1, 2, ..., 9}
	\fill (a.corner \x) circle[radius=1pt];
	\foreach \x in {1, 2, ..., 9}
	\foreach \bend in {-15, 0, 15}
		\path (a.corner \x) edge [bend right=\bend] (c);	
\end{tikzpicture}
\ar[l] & \cdots \ar[l]}
\]
  \caption{($p = 3$) $\lambda = 0$, $\mu = 1$}
    \label{fig:case_p3l1}
\end{figure}

\item
Let $\lambda' \geq 1$ and $\mu \geq 0$.
Then, in fact, we can use the same $h$ as in Subsection \ref{ss:constr} for the same $\lambda'$ and $\mu$, and we only have to take $c = p^{\mu+1}$ (see Figures \ref{fig:case_p3l12'}, \ref{fig:case_p3l3'}, and \ref{fig:case_p3l5'}, which are respectively based on Figures \ref{fig:case_p3l12}, \ref{fig:case_p3l3}, and \ref{fig:case_p3l5}).

%p=3, lambda = 2, mu = 0
\begin{figure}[H]
  \centering
%\Large
\[
\xymatrix{
%0-th
\begin{tikzpicture}[baseline]
\node[circle,fill,inner sep=1pt] (a) at (0,0.5) {};
\node[circle,fill,inner sep=1pt] (b) at (0,-0.5) {};
\foreach \bend in {15, 0,-15}
	\path (a) edge[draw,bend right=\bend] (b);
\path (a) edge[loop above, looseness=30, out=60, in=120] node[above] {\normalsize $\gamma$} (a);
\end{tikzpicture}
&
%1-st
\begin{tikzpicture}[baseline]
	\node[minimum size=4em,regular polygon,regular polygon sides=3] (a) {};
	\foreach \x in {1, 2, ..., 3}
	\fill (a.corner \x) circle[radius=1pt];
	\foreach \x/\y in {1/2, 2/3, 3/1}
	\foreach \bend in {0}
		\path (a.corner \x) edge [bend right=\bend] (a.corner \y);
	\node[circle,fill,inner sep=1pt] (c) at (a.center) {};
	\foreach \x in {1, 2, ..., 3}
	\foreach \bend in {-15, 0, 15}
		\path (a.corner \x) edge [bend right=\bend] (c);	
\end{tikzpicture}
\ar[l] &
%2-nd
\begin{tikzpicture}[baseline]
	\node[minimum size=4em,regular polygon,regular polygon sides=9] (a) {};
	\foreach \x in {1, 2, ..., 9}
	\fill (a.corner \x) circle[radius=1pt];
	\foreach \x/\y in {1/2, 2/3, 3/4, 4/5, 5/6, 6/7, 7/8, 8/9, 9/1}
	\foreach \bend in {0}
		\path (a.corner \x) edge [bend right=\bend]  (a.corner \y);
	\node[circle,fill,inner sep=1pt] (c) at (a.center) {};
	\foreach \x in {1, 2, ..., 9}
	\foreach \bend in {-15, 0, 15}
		\path (a.corner \x) edge [bend right=\bend] (c);	
\end{tikzpicture}
\ar[l] & \cdots \ar[l]}
\]
  \caption{($p = 3$) $\lambda = 2$, $\mu = 0$}
    \label{fig:case_p3l12'}
\end{figure}

%p=3, lambda=4, mu=0
\begin{figure}[H]
  \centering
%\Large
\[
\xymatrix{
%0-th
\begin{tikzpicture}[baseline]
\node[circle,fill,inner sep=1pt] (a) at (0,0.5) {};
\node[circle,fill,inner sep=1pt] (b) at (0,-0.5) {};
\foreach \bend in {15, 0,-15}
	\path (a) edge[draw,bend right=\bend] (b);
\foreach \len in {15,30}
\path (a) edge[loop above, looseness=\len, out=60, in=120] (a);
\path (a) edge[loop above, looseness=45, out=60, in=120] node[above] {\normalsize $\gamma^2, \gamma^2, \gamma$} (a);
\end{tikzpicture}
&
%1-st
\begin{tikzpicture}[baseline]
	\node[minimum size=4em,regular polygon,regular polygon sides=3] (a) {};
	\foreach \x in {1, 2, ..., 3}
	\fill (a.corner \x) circle[radius=1pt];
	\foreach \x/\y in {1/2, 2/3, 3/1}
	\foreach \bend in {15, 0, -15}
		\path (a.corner \x) edge [bend right=\bend] (a.corner \y);
	\node[circle,fill,inner sep=1pt] (c) at (a.center) {};
	\foreach \x in {1, 2, ..., 3}
	\foreach \bend in {-15, 0, 15}
		\path (a.corner \x) edge [bend right=\bend] (c);	
\end{tikzpicture}
\ar[l] &
%2-nd
\begin{tikzpicture}[baseline]
	\node[minimum size=4em,regular polygon,regular polygon sides=9] (a) {};
	\foreach \x in {1, 2, ..., 9}
	\fill (a.corner \x) circle[radius=1pt];
	\foreach \x/\y in {1/2, 2/3, 3/4, 4/5, 5/6, 6/7, 7/8, 8/9, 9/1}
	\foreach \bend in {0}
		\path (a.corner \x) edge [bend right=\bend] (a.corner \y);
	\foreach \x/\y in {1/3, 2/4, 3/5, 4/6, 5/7, 6/8, 7/9, 8/1, 9/2}
	\foreach \bend in {10, -10}
		\path (a.corner \x) edge [bend right=\bend] (a.corner \y);
	\node[circle,fill,inner sep=1pt] (c) at (a.center) {};
	\foreach \x in {1, 2, ..., 9}
	\foreach \bend in {-15, 0, 15}
		\path (a.corner \x) edge [bend right=\bend] (c);	
\end{tikzpicture}
\ar[l] & \cdots \ar[l]}
\]
  \caption{($p = 3$) $\lambda = 4$, $\mu = 0$}
    \label{fig:case_p3l3'}
\end{figure}

%p=3, lambda = 6, mu=0
\begin{figure}[H]
  \centering
%\Large
\[
\xymatrix{
%0-th
\begin{tikzpicture}[baseline]
\node[circle,fill,inner sep=1pt] (a) at (0,0.5) {};
\node[circle,fill,inner sep=1pt] (b) at (0,-0.5) {};
\foreach \bend in {15, 0,-15}
	\path (a) edge[draw,bend right=\bend] (b);
\foreach \len in {15,30,45}
\path (a) edge[loop above, looseness=\len, out=60, in=120] (a);
\path (a) edge[loop above, looseness=60, out=60, in=120] node[above] {\normalsize $\gamma^3, \gamma, \gamma, \gamma$} (a);
\end{tikzpicture}
&
%1-st
\begin{tikzpicture}[baseline]
	\node[minimum size=4em,regular polygon,regular polygon sides=3] (a) {};
	\node[circle,fill,inner sep=1pt] (a1) at (a.corner 1) {};
	\node[circle,fill,inner sep=1pt] (a2) at (a.corner 2) {};
	\node[circle,fill,inner sep=1pt] (a3) at (a.corner 3) {};
\path (a1) edge[loop, looseness=30, out=60, in=120] (a1);
\path (a2) edge[loop, looseness=30, out=180, in=240] (a2);
\path (a3) edge[loop, looseness=30, out=300, in=0] (a3);
\path (0,0) edge[loop, looseness=10, out=45, in=135] (0,0);
	\foreach \x in {1, 2, ..., 3}
	\fill (a.corner \x) circle[radius=1pt];
	\foreach \x/\y in {1/2, 2/3, 3/1}
	\foreach \bend in {15, 0, -15}
		\path (a.corner \x) edge [bend right=\bend] (a.corner \y);
	\node[circle,fill,inner sep=1pt] (c) at (a.center) {};
	\foreach \x in {1, 2, ..., 3}
	\foreach \bend in {-15, 0, 15}
		\path (a.corner \x) edge [bend right=\bend] (c);	
\end{tikzpicture}
\ar[l] &
%2-nd
\begin{tikzpicture}[baseline]
	\node[minimum size=4em,regular polygon,regular polygon sides=9] (a) {};
	\foreach \x in {1, 2, ..., 9}
	\fill (a.corner \x) circle[radius=1pt];
	\foreach \x/\y in {1/2, 2/3, 3/4, 4/5, 5/6, 6/7, 7/8, 8/9, 9/1}
	\foreach \bend in {15,0,-15}
		\path (a.corner \x) edge [bend right=\bend] (a.corner \y);
	\foreach \x/\y in {1/4, 2/5, 3/6, 4/7, 5/8, 6/9, 7/1, 8/2, 9/3}
	\foreach \bend in {0}
		\path (a.corner \x) edge [bend right=\bend] (a.corner \y);
	\node[circle,fill,inner sep=1pt] (c) at (a.center) {};
	\foreach \x in {1, 2, ..., 9}
	\foreach \bend in {-15, 0, 15}
		\path (a.corner \x) edge [bend right=\bend] (c);	
\end{tikzpicture}
\ar[l] & \cdots \ar[l]}
\]
  \caption{($p = 3$) $\lambda = 6$, $\mu = 0$}
    \label{fig:case_p3l5'}
\end{figure}
\end{itemize}

Table \ref{tab:ram} lists the values of $\ord_p(\kappa_{X_n})$ for these examples, computed using SageMath.

\begin{table}[htbp]
\centering
\begin{tabular}{c @{\hspace{2em}} cccccccc}
\toprule
$n$ & \swid{$0$} & \swid{$1$} & \swid{$2$} & \swid{$3$} & \swid{$4$} & \swid{$5$} & \swid{$6$} & $\cdots$ \\
\midrule
Figure \ref{fig:case_p3l1} & $1$ & $3$ & $9$ & $27$ & $81$ & $243$ & $729$ & $\cdots$ \\
Figure \ref{fig:case_p3l12'} & $1$ & $3$ & $5$ & $7$ & $9$ & $11$ & $13$ & $\cdots$ \\
Figure \ref{fig:case_p3l3'} & $1$ & $3$ & $7$ & $11$ & $15$ & $19$ & $23$ & $\cdots$ \\
Figure \ref{fig:case_p3l5'} & $1$ & $3$ & $9$ & $15$ & $21$ & $27$ & $33$ & $\cdots$ \\
\bottomrule
\end{tabular}
\caption{The values of $\ord_p(\kappa_{X_n})$}
\label{tab:ram}
\end{table}

\section*{Acknowledgments}

I am very grateful to Daniel Valli\`{e}res for valuable comments on an earlier version of this paper.
This work is supported by JSPS KAKENHI Grant Number 22K13898.

{
\bibliographystyle{abbrv}
\bibliography{biblio}

\begin{thebibliography}{10}

\bibitem{AMT25}
T.~Adachi, K.~Mizuno, and S.~Tateno.
\newblock Iwasawa theory for weighted graphs.
\newblock {\em Ann. Math. Qu\'{e}.}, 2025.

\bibitem{DLRV24}
C.~Dion, A.~Lei, A.~Ray, and D.~Valli\`eres.
\newblock On the distribution of {I}wasawa invariants associated to
  multigraphs.
\newblock {\em Nagoya Math. J.}, 253:48--90, 2024.

\bibitem{FOO06}
S.~Fujii, Y.~Ohgi, and M.~Ozaki.
\newblock Construction of {$\Bbb Z_p$}-extensions with prescribed {I}wasawa
  {$\lambda$}-invariants.
\newblock {\em J. Number Theory}, 118(2):200--207, 2006.

\bibitem{GV24}
R.~Gambheera and D.~Valli\`eres.
\newblock Iwasawa theory for branched {$\Bbb{Z}_p$}-towers of finite graphs.
\newblock {\em Doc. Math.}, 29(6):1435--1468, 2024.

\bibitem{Gon22}
S.~R. Gonet.
\newblock Iwasawa {T}heory of {J}acobians of graphs.
\newblock {\em Algebr. Comb.}, 5(5):827--848, 2022.

\bibitem{Iwa59}
K.~Iwasawa.
\newblock On {$\Gamma $}-extensions of algebraic number fields.
\newblock {\em Bull. Amer. Math. Soc.}, 65:183--226, 1959.

\bibitem{Kata_21}
T.~Kataoka.
\newblock Fitting ideals of {J}acobian groups of graphs.
\newblock {\em Algebr. Comb.}, 7(3):597--625, 2024.

\bibitem{Kata_31}
T.~Kataoka.
\newblock Kida's formula for graphs with ramifications.
\newblock {\em Acta Arith.}, 221(3):271--291, 2025.

\bibitem{LV23}
A.~Lei and D.~Valli\`eres.
\newblock The non-{$\ell$}-part of the number of spanning trees in abelian
  {$\ell$}-towers of multigraphs.
\newblock {\em Res. Number Theory}, 9(1):Paper No. 18, 16, 2023.

\bibitem{MV23}
K.~McGown and D.~Valli\`eres.
\newblock On abelian {$\ell$}-towers of multigraphs {II}.
\newblock {\em Ann. Math. Qu\'e.}, 47(2):461--473, 2023.

\bibitem{MV24}
K.~McGown and D.~Valli\`eres.
\newblock On abelian {$\ell$}-towers of multigraphs {III}.
\newblock {\em Ann. Math. Qu\'{e}.}, 48(1):1--19, 2024.

\bibitem{Oza04}
M.~Ozaki.
\newblock Construction of {$\bold Z_p$}-extensions with prescribed {I}wasawa
  modules.
\newblock {\em J. Math. Soc. Japan}, 56(3):787--801, 2004.

\bibitem{Val21}
D.~Valli\`eres.
\newblock On abelian {$\ell$}-towers of multigraphs.
\newblock {\em Ann. Math. Qu\'e.}, 45(2):433--452, 2021.

\end{thebibliography}
}

\end{document}